\documentclass[reqno]{amsart}
\usepackage{hyperref}
\usepackage{enumerate,cite}
\newcommand{\R}{\mathbb{R}}

\newcommand{\dom}{\mathrm{Dom}\;}
\newcommand{\interior}{\mathrm{Int}\;}
\newcommand{\ima}{\mathrm{Im}\;}
\newcommand{\Ker}{\mathrm{Ker}\;}

\begin{document}
\title[ Positive periodic solution of computer virus model]
{Sufficient conditions for existence of positive periodic solution of
a generalized nonresident computer virus model}

\author[A. Coronel, F. Huancas, M. Pinto]
{An{\'\i}bal\ Coronel, Fernando Huancas, Manuel Pinto}

\address{An{\'\i}bal Coronel \newline
GMA, Departamento de Ciencias B\'asicas,
Facultad de Ciencias, Universidad del B\'{\i}o-B\'{\i}o,
Campus Fernando May, Chill\'{a}n, Chile}
\email{acoronel@ubiobio.cl}

\address{Fernando Huancas \newline
GMA, Departamento de Ciencias B\'asicas,
Facultad de Ciencias, Universidad del B\'{\i}o-B\'{\i}o,
Campus Fernando May, Chill\'{a}n, Chile}
\email{fihuanca@gmail.com}

\address{Manuel Pinto \newline
Departamento de Matem\'aticas, Facultad de Ciencias,
Universidad de Chile, Chile,}
\email{pintoj.uchile@gmail.cl}


\begin{abstract}
 In this paper, we introduce a nonresident computer virus model
and prove the existence of at least one positive periodic solution. 
The proposed model is based on a biological approach and is
obtained by considering that all rates
(rates that the computers are disconnected from the Internet,
the rate that the computers are cured, etc)
are  time dependent real functions.
Assuming that the initial condition is a positive vector and  the
coefficients are positive $\omega-$periodic
and applying the topological degree arguments we deduce that 
generalized nonresident computer virus model has
at least one positive $\omega-$periodic solution. The proof
consists of  two big parts. First,
an appropriate change of variable which conserves the periodicity property and implies
the positive behavior. Second, a reformulation of transformed system as an operator
equation which is analyzed  by applying the continuation theorem of
the coincidence degree theory.
\end{abstract}

\maketitle
\numberwithin{equation}{section}
\newtheorem{definition}{Definition}[section]
\newtheorem{theorem}{Theorem}[section]
\newtheorem{lemma}[theorem]{Lemma}
\newtheorem{proposition}[theorem]{Proposition}
\allowdisplaybreaks

\section{Introduction}
In the last decades, the study of widespread infection of the computers connected
to internet has attracted the interest of several researchers
(see for instance
\cite{chen_2015,dong_2012,gan_2014_nd,gan_2014_aaa,gen_2014,
gan_2014_cnsns,gan_2013,gan_2012,murray_1998,muroya_2015,
muroya_2014_ijcm,muroya_2014,nyamordi_2012}). It is well known
that the first appearance of computer viruses occurs in 1980 and the formalization 
of the computer virus problem and the related concepts were developed and presented
independently by Solomon~\cite{solomon_1993} and Cohen~\cite{cohen_thesis,cohen_1987}. 
Afterwards, many studies with diverse 
themes, but focused on the problem. In a broad sense, the unification of these 
works is that the main effort is the development of mathematical models. However,
the difference is the theoretical basis of the models are disperse and diverse, since
are constructed following the existing analogies with other traditional approaches 
(Markov chains \cite{amador_2013,billings_2002}, 
graph theory \cite{goldberg_1998,goldberg_1996}, dynamical systems 
\cite{yang_preprint,yang_2013,han_2010,hu_2015,li_2008,peng_2013,
ren_2012,ren_2012_chaos,ren_2012_narwa,yang_2012,_2012,zhang_2015,zhu_2012}) 
of the epidemiological analysis on large populations.

Kephart and White, following the ideas suggested by Cohen~\cite{cohen_1987}
and Murray~\cite{murray_1998},
introduce the first propagation model of computer virus in~\cite{kephart_ieee_proc_1993}
(see also \cite{kephart_ieee_1993}). The assumptions 
given in~\cite{kephart_ieee_proc_1993}
permits the dynamic modelling under deterministic models based on differential
equations and stochastic models based on Markov chains. However, 
in~\cite{kephart_ieee_proc_1993} there is not a clear
distinction when stochastic or deterministic approaches are the most preferable to
describe the dynamics of computer virus spreading. Since the work of Billings
et al.~\cite{billings_2002} and Amador and Artalejo~\cite{amador_2013}, 
these facts are clearly disembowel and, roughly speaking, 
the stochastic models are more accurate in the case of local network of small or moderate
size and the deterministic models in the case of the Internet. More specific differences 
on stochastic and deterministic points of view are listed on~\cite{amador_2013}.

In this paper, we are interested in the deterministic model proposed by 
for nonresident virus propagation. We recall that a nonresident virus
is conceptually defined as the virus which does not store or execute itself 
from the computer memory. Yang et al.~\cite{yang_preprint,yang_2013} 
proposed the model considering two 
consecutive phases: (i) {\it The latent phase,} in which the virus has not
yet been loaded into memory  and the
infected computer can infect other computers through file transmission or web browsing;
and 
(ii) {\it The attack  phase,} in which the virus has been run and
 the infected computer can infect other computers through infecting 
new hosts when those files are accessed by other programs or the operating
system itself. Indeed, following the ideas and the notation of \cite{yang_2013},
let us consider that 
the varying total numbers of computers in the network are further divided
at any time $t$ into three compartments denoted by
$S(t),L(t)$ at $A(t)$. Here, 
$S(t)$ denotes the average numbers 
of uninfected computers (susceptible computers) at time $t$,
$L(t)$ denotes the average numbers of infected computers (latent computers) 
in which viruses are not yet
loaded in their memory at time t; and 
$A(t)$ denotes the average numbers of infected computers (infectious computers)
in which viruses are
located in memory  at time t. Thus, assuming that the following hypotheses
\cite{muroya_2015}:
\begin{enumerate}
  \item[(H1)] All newly accessed computers are virus-free.
\item[(H2)] All viruses staying in computers are nonresident.
\item[(H3)] External computers are accessed to the Internet at 
positive constant number $b$ at each time~$t$, and the uninfected, latent
computers and infectious computers of internal computers are 
disconnected from the Internet also at rates $\mu_1$, $\mu_2$, and
 $\mu_3$, respectively at each time~$t$.
\item[(H4)] Users of latent computers cannot perceive
the existence of virus, so latent computers cannot get cured.
\item[(H5)] The numbers of internal computers infected
at time $t$ increases by $\beta_1SL+\beta_2SA$, where  $\beta_1$ and  $\beta_2$
are positive constants.
\item[(H6)] Nonresident viruses within latent computers are
loaded into memory at positive constant rate $\gamma_1$, 
and nonresident viruses
within infectious computers transfer control to the 
application program at positive constant rate $\gamma_2$.
\item[(H7)] Latent computers are cured at positive 
constant rate $\gamma_1$, whereas infected computers are cured at positive constant
rate $\gamma_2$.
\item[(H8)] $\alpha_1$ and $\alpha_2$ are the rates of 
nonresident viruses within latent computers are loaded into
memory and nonresident viruses
within infectious computers transfer control to the application program, respectively.
\end{enumerate}
holds, the following ordinary differential equation system
\begin{subequations}
\label{eq:virus_model_original}
\begin{eqnarray}
\frac{dS(t)}{dt}
&=& b- \mu_1S(t)-\beta_1S(t)L(t)-
	\beta_2S(t)A(t)+\gamma_1L(t)
	+\gamma_2A(t),
	\label{eq:virus_model_original:S}\\
\frac{dL(t)}{dt}
&=&\beta_1S(t)L(t)+ \beta_2S(t)A(t)
	+\alpha_2A(t)-[\mu_2
	+\alpha_1+\gamma_1]L(t),
	\label{eq:virus_model_original:L}\\
\frac{dA(t)}{dt}
&=&\alpha_1L(t)-[\mu_3
+\alpha_1+\gamma_2]A(t),
\label{eq:virus_model_original:A}
\end{eqnarray} 
\end{subequations}
is deduced as mathematical model for nonresident computer virus 
propagation  with varying total numbers of computers in the
network.

Recently, in \cite{muroya_2015} considering that the 
initial condition $(S(0),L(0),A(0))\in\mathbb{R}^3_+$
and all parameters $b,\alpha_i,\beta_i,\mu_i$ and $\gamma_i$ with $i=1,2,3,$ are 
all real positive constants with $\mu_1\le \min\{\mu_2,\mu_3\}$, the authors
prove the global stability of \eqref{eq:virus_model_original} establishing that
there is two dynamical globally asymptotic stability
possibilities: converging to an uninfected or infected equilibrium,
depending if the reproduction number $R_0$ is such that $R_0\le 1$ 
or $R_0>1$, respectively. Moreover,  other properties
of the model \eqref{eq:virus_model_original} and close models
based on biological models
are studied in several works  
and also some modifications of the model
are proposed recently (see for instance 
\cite{ren_2014,ren_2013,ren_2014_aaa,muroya_2014_ijcm,
muroya_2014,nyamordi_2012,song_2014,wierman_2004,yang_2014,
zhang_2013,zhu_amc_2012}). 
However, there are some other properties of the dynamical phenomena which are 
well established  in epidemiological models, but are not analyzed yet
for \eqref{eq:virus_model_original}.
For instance, in the best of our knowledge, there is not
a result for the existence 
of positive periodic solutions 
(see \cite{hiv_reference} for a Hepatitis model). 

In this paper, we consider all rates are time dependent, i.e.
the assumptions on (H3), (H5)-(H8) are more general in the 
sense that the parameters $b,\alpha_i,\beta_i,\mu_i$ and $\gamma_i$
with $i=1,2,3,$ are time dependent  real functions. These new considerations
motivate the following generalized model:
\begin{subequations}
\label{eq:virus_model}
\begin{eqnarray}
\frac{dS(t)}{dt}
&=& b(t)- \mu_{1}(t)S(t)-\beta_{1}(t)S(t)L(t)
\nonumber\\
&&-
	\beta_{2}(t)S(t)A(t)+\gamma_{1}(t)L(t)
	+\gamma_{2}(t)A(t),
	\label{eq:virus_model:S}\\
\frac{dL(t)}{dt}
&=&\beta_{1}(t)S(t)L(t)+ \beta_{2}(t)S(t)A(t)
\nonumber\\
&&
	+\alpha_{2}(t)A(t)-[\mu_{2}(t)
	+\alpha_{1}(t)+\gamma_{1}(t)]L(t),
	\label{eq:virus_model:L}\\
\frac{dA(t)}{dt}
&=&\alpha_{1}(t)L(t)-[\mu_{3}(t)
+\alpha_{1}(t)+\gamma_{2}(t)]A(t).
\label{eq:virus_model:A}
\end{eqnarray} 
\end{subequations}
Now, in order to understand the dynamics we study the 
existence of positive periodic solution for \eqref {eq:virus_model}.

The main result of the paper is given by the following theorem:
\begin{theorem}
\label{teo:main}
Assume that the coefficients of the system \eqref{eq:virus_model}
satisfy the following hypothesis:
\begin{eqnarray}
\left.
\begin{array}{l}
\mbox{The initial condition $(S(0),L(0),A(0))\in\R^3_+$ and the coefficient functions}
\\
\mbox{$b,\mu_1,\beta_1,\beta_2,\gamma_1,\gamma_2,\alpha_1$ and $\alpha_2$
are positive, continuous, $\omega$-periodic on $[0,\omega]$ and}
\\
\hspace{2cm}
\displaystyle
\max_{t\in [0,\omega]}\left[\frac{\alpha_1(\alpha_2+\gamma_2)}
{(\alpha_1+\mu_2)}\right](t)
\le 
\min_{t\in [0,\omega]}(\mu_3+\alpha_2+\gamma_2)(t).
\end{array}
\right\}
\label{eq:hipot_H}
\end{eqnarray}
Then, the system \eqref {eq:virus_model}
has at least one positive $\omega$-periodic solution.
\end{theorem}
\noindent
To prove the Theorem~\ref{teo:main} we apply the coincidence degree  theory.
The proof is self-contained and presented in section~\ref{sec:proof}, by 
introducing several lemmas which implies that the hypotheses of continuation
theorem are valid in this context. Moreover, 
it is  worthwhile to remark that another important result of the paper is the
a priory estimates given on Theorem~\ref{teo:contradiction}, which is useful to get
the contradiction in one of the steps of  the proof of Theorem~\ref{teo:main}.

\section{Proof of main result: Theorem~\ref{teo:main}}
\label{sec:proof}

The proof is given in the following  two steps
\begin{enumerate}
 \item[(i)] We introduce a change of variable which ``conservs'' the periodicity property
 and implies the positive behavior (see Theorem~\ref{prop:sist_eqiv}).
 \item[(ii)] We prove the existence of periodic solutions of the transformed system
 by applying  the topological degree arguments.
\end{enumerate}

\subsection{Reformulation of original system  \eqref{eq:virus_model}}

\begin{theorem}
\label{prop:sist_eqiv}
Assume that the  system \eqref{eq:virus_model}
has a solution. Then, 
$\{S,L,A\}$ is a solution of the system \eqref{eq:virus_model}
if and only if $\{S^{\ast},L^{\ast},A^{\ast}\}$ defined as
follows
\begin{eqnarray}
S(t)=\exp(S^{\ast}(t)),
\quad
L(t)=\exp(L^{\ast}(t)),
\quad
A(t)=\exp(A^{\ast}(t)),
\label{eq:change_of_variable}
\end{eqnarray}
is a solution of the following system
\begin{subequations}
\label{eq:virus_model_equiv}
\begin{eqnarray}
\frac{dS^{\ast}(t)}{dt} 
&=& b(t)\exp({-S^{\ast}(t)})
-\beta_{1}(t)\exp({L^{\ast}(t)})-\beta_{2}(t)\exp({A^{\ast}(t)})
\nonumber\\
&&
+\gamma_{1}(t)\exp\Big({L^{\ast}(t)-S^{\ast}(t)}\Big)
+\gamma_{2}(t)\exp\Big({A^{\ast}(t)-S^{\ast}(t)}\Big)- \mu_{1}(t),
\qquad\mbox{${}$}
\label{eq:virus_model_equiv:s}\\
\frac{dL^{\ast}(t)}{dt}
&=&\beta_{1}(t)\exp({S^{\ast}(t)})
+ \beta_{2}(t)\exp\Big({S^{\ast}(t)+A^{\ast}(t)-L^{\ast}(t)}\Big)
\nonumber\\
&&
+\alpha_{2}(t)\exp\Big({A^{\ast}(t)-L^{\ast}(t)}\Big)
-[\mu_{2}(t)
+\alpha_{1}(t)+\gamma_{1}(t)],
\label{eq:virus_model_equiv:l}\\
\frac{dA^{\ast}(t)}{dt}
&=&\alpha_{1}(t)\exp\Big({L^{\ast}(t)-A^{\ast}(t)}\Big)
-[\mu_{3}(t)+\alpha_{2}(t)+\gamma_{2}(t)].
\label{eq:virus_model_equiv:a}
\end{eqnarray}
\end{subequations}
In particular, we have that the following two assertions are valid:
\begin{enumerate}[(a)]
\item If the solution of  the system 
\eqref{eq:virus_model_equiv} 
is $\omega$-periodic, then the solution of \eqref{eq:virus_model}
is $\omega$-periodic.
\item If the system  \eqref{eq:virus_model_equiv} has a solution, then
\eqref{eq:virus_model} has a positive solution.
\end{enumerate}
\end{theorem}

\begin{proof}
We follow the proof of  \eqref{eq:virus_model_equiv} by differentiation
of the new variables defined on \eqref{eq:change_of_variable}. To be more
precise, let us consider that $\{S,L,A\}$ is a solution
of \eqref{eq:virus_model}, then we can prove that
$\{S^{\ast},L^{\ast},A^{\ast}\}$ is a solution
of \eqref{eq:virus_model_equiv} by using \eqref{eq:change_of_variable}
and by multiplying the equations 
\eqref{eq:virus_model:S}, \eqref{eq:virus_model:L} and \eqref{eq:virus_model:A}
by $\exp(-S^{\ast}(t)),$ $\exp(-L^{\ast}(t))$ and $\exp(-A^{\ast}(t))$,
respectively. Conversely, if we assume that 
$\{S^{\ast},L^{\ast},A^{\ast}\}$ is a solution
of \eqref{eq:virus_model_equiv}, then by \eqref{eq:change_of_variable}
and differentiation we deduce that $\{S,L,A\}$ is a solution
of \eqref{eq:virus_model}. Indeed, to verify that  
of $S$ satisfies \eqref{eq:virus_model:S}, by \eqref{eq:virus_model_equiv:s} and \eqref{eq:change_of_variable}
 we have that
\begin{eqnarray*}
\frac{dS(t)}{dt}
&=&\exp({S^{\ast}(t)})\frac{dS^{\ast}(t)}{dt}
\\
&=&
\exp({S^{\ast}(t)})
\Big[
b(t)\exp({{-S^{\ast}(t)}})-
\mu_{1}(t)-\beta_{1}(t)\exp({{L^{\ast}(t)}})
\\
&&
-\beta_{2}(t)\exp({{A^{\ast}(t)}})
+\gamma_{1}(t)\exp({{L^{\ast}(t)}-{S^{\ast}(t)}})
+\gamma_{2}(t)\exp({{A^{\ast}(t)}-{S^{\ast}(t)}})\Big]
\\
&=&
b(t)- \mu_{1}(t)S(t)-\beta_{1}(t)S(t)L(t)-
	\beta_{2}(t)S(t)A(t)+\gamma_{1}(t)L(t)
	+\gamma_{2}(t)A(t).
\end{eqnarray*}
Similarly, we can verify that \eqref{eq:virus_model:L} and \eqref{eq:virus_model:A}
are satisfied.

Assuming that the functions
$\{S^{\ast},L^{\ast},A^{\ast}\}$ are $\omega$-periodic, then 
by \eqref{eq:change_of_variable} we can get that 
the functions $\{S^{\ast},L^{\ast},A^{\ast}\}$ are $\omega$-periodic,
since (for instance in the case of $S$) we have that
\begin{eqnarray*}
S(t+\omega)=\exp({S^{\ast}(t+\omega)})=\exp({S^{\ast}(t)})=S(t).
\end{eqnarray*}
Then (a) is proved. Now, we note that
(b) is a straightforward consequence of the definition
of $\{S^{\ast},L^{\ast},A^{\ast}\}$ given in \eqref{eq:change_of_variable}.
\end{proof}

\subsection{Existence of a periodic solutions for the system 
\eqref{eq:virus_model_equiv}}

In this subsection we prove the following theorem:
\begin{theorem}
\label{teo:exitencia_equiv}
Assume that the coefficients of the system \eqref{eq:virus_model_equiv}
satisfy the following hypothesis~\eqref{eq:hipot_H}.
Then, the system \eqref {eq:virus_model_equiv}
has at least one $\omega$-periodic solution.
\end{theorem}

\subsubsection{The topological degree notation, concepts and results.}
$ $

In order to analyze the system \eqref{eq:virus_model_equiv}
we apply the topological degree arguments.
Indeed, for completeness of the presentation, 
we recall some notation, concepts and results of these theory
(see \cite{mawhin_book} for details).

\begin{definition}
\label{def:L_is_fredholm}
Let $X$ and $Y$ be normed vector spaces and  
$L:\dom L\subset X\to Y$  a linear operator. Then,  $L$ is called a
Fredholm operator of index zero, if the following assertions
\begin{eqnarray}
 {\rm dim}(\Ker L)= {\rm codim}(\ima L) <\infty
 \qquad
 \mbox{and}
 \qquad
 \mbox{$\ima L$ is closed in $Y$},
 \label{def:L_is_fredholm:1}
\end{eqnarray}
are valid.
\end{definition}

\begin{proposition}
\label{prop:projectorsPQ}
Let $X$ and $Y$ be normed vector spaces and  
$L:\dom L\subset X\to Y$ a linear operator.
If $L$ is a Fredholm mapping of index zero, then 
\begin{enumerate}[(i)]
 \item There are two
continuous projectors $P:X \to X$ and $Q: Y \to Y$ such that
$\ima P=\Ker L$ and $\ima L=\Ker Q=\ima (I-Q)$.

    \item  $L_{P}:=L |_{\dom L
\cap \Ker P}:(I-P)X\to \ima L$ is invertible and its inverse
is denoted by $K_{P}$.

\item There is an isomorphism $J:\ima Q\to \Ker L$.

\end{enumerate}

\end{proposition}

\begin{definition}
\label{def:L_compact}
Let $X$ and $Y$ be normed vector spaces and  
$L:\dom L\subset X\to Y$ a Fredholm mapping of index zero.
Let $P:X \to X$ and $Q: Y \to Y$ be two
continuous projectors  such that
$\ima P=\Ker L$ and $\ima L=\Ker Q=\ima (I-Q)$.
Let us consider $N:X\to Y$  a continuous operator and
$\Omega\subset X$ an open bounded set. 
Then, $N$ is called $L-$compact on $\overline{\Omega}$
if $QN(\overline{\Omega})$ is a  bounded set  and the operator
$K_{P}(I-Q)N$ is compact on $\overline{\Omega}$.
\end{definition}

\begin{theorem}
\label{teo:teocontinuacion}
Assume that $(X,\|.\|_X)$ and $(Y,\|.\|_Y)$
are two Banach spaces and $\Omega$ is an open bounded set.
Consider that $L:\dom L\subset X\to Y$ be a Fredholm mapping of 
index zero and $N:X\to Y$ be $L-$compact
on $\overline{\Omega}$. If the following hypotheses 
\begin{enumerate}
    \item[(C$_1$)] $ Lx\neq \lambda Nx $
    for each $(\lambda,x) \in (0,1)\times (\partial \Omega \cap \dom L)$.
    \item[(C$_2$)] $QNx\neq 0$
    for each $x \in \partial \Omega \cap \Ker L.$
    \item[(C$_3$)] $deg (JQN,\Omega \cap \Ker L,0)\neq 0$.
\end{enumerate}
are valid.
Then the operator equation $Lx= Nx$ has at least one
solution in $\dom L\cap \overline{\Omega}.$
\end{theorem}

\begin{definition}
\label{def:degree}
Let $\Omega\subset\R^n$ be an open bounded set, 
$f\in C^1(\Omega,\R^n)\cap C(\overline{\Omega},\R^n)$ and 
$y\in \R^n\backslash f(\partial\Omega\cup N_f)$, i.e.
$y$ is a regular value of $f$. Here, $N_f=\{x\in\Omega:J_f(x)=0\}$
the critical set of $f$ and $J_f$ the Jacobian of $f$ at $x$.
Then, the degree $\mathrm{deg}\{f,\Omega,y\}$ is defined by
\begin{eqnarray*}
\mathrm{deg}\{f,\Omega,y\}=\sum_{x\in f^{-1}(y)}\mathrm{sgn} J_f(x),
\end{eqnarray*}
with the agreement that $\sum \phi =0$.

\end{definition}

\subsubsection{Reformulation of \eqref{eq:virus_model_equiv} as an operator equation.}
$ $

Let us consider that the  $X$ and $Y$ are two appropriate normed vector spaces 
and define the operators $L:\dom L\subset X \to Y$ and 
$N:X\rightarrow Y$
\begin{eqnarray}
 L\Big((x_{1},x_{2},x_{3})^{T}\Big)
&=&\left(\frac{dx_{1}}{dt},\frac{dx_{2}}{dt},\frac{dx_{3}}{dt}\right)^{T}
\label{eq:operator_L}
\\
N\Big((x_{1},x_{2},x_{3})^{T}\Big)
& = &
\Big(\mathcal{N}_1,\mathcal{N}_2,\mathcal{N}_3\Big)^{T}
\label{eq:operator_N}
\end{eqnarray}
where
\begin{eqnarray}
 \mathcal{N}_1(t)
 &=&b(t)\exp({-x_{1}(t)})- \mu_{1}(t)
 -\beta_{1}(t)\exp({x_{2}(t)})
 -\beta_{2}(t)\exp({x_{3}(t)})
 \nonumber\\
 &&
  +\gamma_{1}(t)\exp({x_{2}(t)-x_{1}(t)})
  +\gamma_{2}(t)\exp({x_{3}(t)-x_{1}(t)})
  \qquad{}\label{eq:operator_N:1}\\
  \mathcal{N}_2(t)
 &=&\beta_{1}(t)\exp({x_{1}(t)})+ \beta_{2}(t)\exp({x_{1}(t)+x_{3}(t)-x_{2}(t)})
  \nonumber\\
 &&
  +\alpha_{2}(t)\exp({x_{3}(t)-x_{2}(t)})-[\mu_{2}(t)+\alpha_{1}(t)+\gamma_{1}(t)] 
  \label{eq:operator_N:2}\\
  \mathcal{N}_3(t)
 &=&
  \alpha_{1}(t)\exp({x_{2}(t)-x_{3}(t)})-[\mu_{3}(t)+\alpha_{2}(t)+\gamma_{2}(t)]. 
  \label{eq:operator_N:3}
\end{eqnarray}
Then, we can rewrite the system \eqref{eq:virus_model_equiv}
as the following operator equation 
\begin{eqnarray}
L\Big((S^{\ast},L^{\ast},A^{\ast})^T\Big)
=
N\Big((S^{\ast},L^{\ast},A^{\ast})^T\Big), 
\quad (S^{\ast},L^{\ast},A^{\ast})\in \dom L\subset X,
\label{eq:operator_equation}
\end{eqnarray}
where the Banach spaces $X$ and $Y$ coincides and are given by
\begin{eqnarray}
&&X=Y=\Big\{(x_{1},x_{2},x_{3})^{T}\in
C(\mathbb{R},\mathbb{R}^{3})\quad:\quad x_{i}(t+\omega)=x_{i}(t),
\;i\in\{1,2,3\},\;
\nonumber\\
&&
\hspace{2.5cm}
\Big\|(x_{1},x_{2},x_{3})^{T}\Big\|
=\sum_{i=1}^3\max_{t \in
[0,\omega]}|x_{i}(t)|<\infty
\Big\}.
\hspace{1cm}
\label{eq:X:space}
\end{eqnarray}
The spaces in \eqref{eq:X:space} are the more appropriate,
since  we are interested in the study of $\omega$-periodic solutions.
Thus, the proof of existence of positive periodic
solutions for \ref{eq:virus_model_equiv}
is reduced the  application of 
Theorem~\ref{teo:teocontinuacion} to equation
\eqref{eq:operator_equation}, i.e. is reduced
to prove that the operators 
$L$ and $N$ defined on \eqref{eq:operator_L} and \eqref{eq:operator_N} 
satisfy all the hypotheses of Theorem~\ref{teo:teocontinuacion}.

\subsubsection{$L$ defined on \eqref{eq:operator_L} is a Fredholm operator of index zero}
$ $

\begin{lemma}
\label{lem:L_is_fredholm}
Let us consider $X$ and $Y$ the Banach spaces given on \eqref{eq:X:space}.
Then $L:\dom L\subset X\to Y$  defined on \eqref{eq:operator_L}
is a Fredholm operator of index zero.
\end{lemma}

\begin{proof}
 
To proof the Lemma we need to verify that the linear operator $L$ 
defined on \eqref{eq:operator_L} satisfies 
the Definition~\ref{def:L_is_fredholm}. Indeed,
we calculate $\Ker L$ and $\ima L$. First,
if we consider that $(S^{\ast},L^{\ast},A^{\ast})^{T}\in \Ker L $,
then we have that 
$(S^{\ast},L^{\ast},A^{\ast})(t)=(s_0,l_0,a_0) \in \mathbb{R}^3$
for all $t\ge t_0$ such that $(S^{\ast},L^{\ast},A^{\ast})(t_0)=(s_0,l_0,a_0)$,
which naturally implies that
\begin{eqnarray}
\Ker L \cong\mathbb{R}^{3}.
 \label{eq:kenelL}
\end{eqnarray}
Second, if we choose  
$(S^{\ast},L^{\ast},A^{\ast})^{T}\in \ima L $,
we have that there is $(S,L,A)\in \dom L $
such that $L\Big((S,L,A)^{T}\Big)=(S^{\ast},L^{\ast},A^{\ast})^{T}$.
Then, by the definition of $L$ and the $w$-periodic behavior
of the functions $\{S,L,A\}$, we deduce
that
\begin{eqnarray*}
\int_{t}^{t+\omega}(S^{\ast}(\tau),L^{\ast}(\tau),A^{\ast}(\tau))^{T}d\tau=0,
\mbox{ for each $t\ge t_0$}
\end{eqnarray*}
or equivalently, by the $\omega$-periodicity of $(S^{\ast},L^{\ast},A^{\ast}),$ 
we have that
\begin{eqnarray}
\ima L=\left\{(S^{\ast},L^{\ast},A^{\ast})\in Y\quad:\quad
\int_{0}^{\omega}(S^{\ast}(\tau),L^{\ast}(\tau),A^{\ast}(\tau))^{T}d\tau=0\right\}.
\label{eq:imagen_L}
\end{eqnarray}
Now, by the elementary results of linear algebra we have the following
three relations $X\cong \ima L\oplus (X/\ima L),$ $X\cong \Ker L\oplus (X/\Ker L),$
and $\ima L\cong X/\Ker L$. Then, $\Ker L\cong X/\ima L$
and we deduce that ${\rm dim}(\Ker L)= {\rm codim}(\ima L)=3$.
Thus the first condition in \eqref{def:L_is_fredholm:1} 
is satisfied. Moreover,
if we consider the linear continuous
mapping $F:\ima L\subset Y \to \mathbb{R}^{3}$
defined as follows
\begin{eqnarray*}
 F\Big((x_{1},x_{2},x_{3})^{T}\Big)
 =
\left (
\int_{0}^{\omega}x_{1}(\tau)d\tau,
\int_{0}^{\omega}x_{2}(\tau)d\tau,
\int_{0}^{\omega}x_{3}(\tau)d\tau
\right),
\end{eqnarray*}
we note that 
$F^{-1}\Big((0,0,0)^{T}\Big)= \ima L$, then  $\ima L$ is a closed
set of $Y$ and second condition in  \eqref{def:L_is_fredholm:1} is 
also verified. 
\end{proof}

\subsubsection{Construction of the projectors $P,Q$ and
the operator $K_{P}$}
$ $


In this subsection we construct the
projectors $P,Q$ and $K_{P}$  asociated to $L$, the 
Fredholm operator of index zero defined on \eqref{eq:operator_L},
and satisfying the Proposition~\ref{prop:projectorsPQ}.
Indeed, if we consider that $P$
and $Q$ are defined as follows
\begin{eqnarray}
P\left((x_{1},x_{2},x_{3})^{T}\right)
=
Q\left((x_{1},x_{2},x_{3})^{T}\right)
=
\frac{1}{\omega}
 \int_{0}^{\omega}
 (x_{1}(\tau),x_{2}(\tau),x_{3}(\tau))^{T} d\tau
  \label{eq:operators_P_Q}
\end{eqnarray}
for any $(x_{1},x_{2},x_{3})\in X$.
Then, we can note that
\begin{enumerate}[(a)]
\item \underline{$\Ker L=\ima P$}. 
We prove this fact by double inclusion argument.
First, if $(S^{\ast},L^{\ast},A^{\ast})^{T} \in \Ker L$,
then by \eqref{eq:kenelL}, we have that 
$(S^{\ast},L^{\ast},A^{\ast})(t)=(s_0,l_0,a_0)\in\mathbb{R}^3$
for all $t\ge t_0,$
which implies that $(S^{\ast},L^{\ast},A^{\ast})\in\ima P$,
since for each $(s_0,l_0,a_0)\in\mathbb{R}^3$ there is
$(s_0,l_0,a_0)\in X$ such that $ P\Big(s_0,l_0,a_0\Big)=(s_0,l_0,a_0),$
i.e. $\Ker L \subset \ima P$. Conversely, if 
$(S^{\ast},L^{\ast},A^{\ast})^{T} \in \ima P$ we have that there is
$(z_{1},z_{2},z_{3}) \in X$ such that 
$P((z_{1},z_{2},z_{3})^{T} )=(S^{\ast},L^{\ast},A^{\ast})^{T}$.
Then by \eqref{eq:operators_P_Q} we follow that
$\omega^{-1}\int_{0}^{\omega}
(z_{1}(\tau),z_{2}(\tau),z_{3}(\tau))^{T} d\tau=(S^{\ast},L^{\ast},A^{\ast})$
and by differentiation we deduce that $L(S^{\ast},L^{\ast},A^{\ast})=(0,0,0)$,
i.e. $\ima P\subset\Ker L.$

\item  \underline{$\Ker Q=\ima L $}.
We follow the proof of this equality by application of \eqref{eq:imagen_L}
and \eqref{eq:operators_P_Q}. Indeed, by \eqref{eq:operators_P_Q}
we have that 
$(S^{\ast},L^{\ast},A^{\ast})^{T} \in \Ker Q$ is equivalent to 
$\int_{0}^{\omega}(S^{\ast},L^{\ast},A^{\ast})^{T}(\tau)d\tau =(0,0,0)$
and by \eqref{eq:imagen_L} this equality means that 
$(S^{\ast},L^{\ast},A^{\ast})^{T} \in \ima L.$

\item  \underline{$\ima (I-Q)=\ima L $}.
The proof follows by a double inclusion argument. If   
$(S^{\ast},L^{\ast},A^{\ast})^{T} \in\ima (I-Q) $, then 
there is $(z_1,z_2,z_3)\in X$ such that 
$(I-Q)\Big((z_1,z_2,z_3)^{T}\Big)=(S^{\ast},L^{\ast},A^{\ast})^{T}$.
Then, by integration and the definition of $Q$,  given on \eqref{eq:operators_P_Q},  
we have that
\begin{eqnarray*}
&&\int_0^{\omega}(S^{\ast},L^{\ast},A^{\ast})^{T}(\tau)d\tau
\\
&&
\hspace{1.5cm}
=\int_0^{\omega}\left(
 (z_1,z_2,z_3)^{T}(\tau)
 -\frac{1}{\omega}\int_0^{\omega}(z_1,z_2,z_3)^{T}(m)dm
 \right)d\tau
=(0,0,0).
\end{eqnarray*}
Thus, by \eqref{eq:imagen_L} we deduce that $(S^{\ast},L^{\ast},A^{\ast})^{T}\in \ima L,$
i.e. $\ima (I-Q)\subset\ima L.$ The proof of $\ima L\subset\ima (I-Q)$
is analogous.

\item  \underline{Operators $K_{P}$ and $L_{P}.$}
The operator $L_{P}:\dom L\cap\Ker P\to \ima L$
is the restriction of $L$ to $\dom L\cap\Ker P,$ i.e.
$L_{P}=L$ on $\dom L\cap\Ker P$.
The operator $K_{P}$, the inverse of $L_{P}$, 
is defined as follows
\begin{eqnarray}
&& K_{P}\Big((x_{1},x_{2},x_{3})^{T}\Big)(t)
\nonumber\\
&&
\hspace{1.0cm}
=\int_0^t(x_{1},x_{2},x_{3})^{T}(\tau)d\tau
 -\frac{1}{\omega}\int_0^\omega\int_0^{\eta}(x_{1},x_{2},x_{3})^{T}(m)dmd\eta.
 \label{eq:operator_kp}
\end{eqnarray}
We can prove that $K_{P}$ is the inverse of $L_{P}$ 
by application of the following fundamental  relation
\begin{eqnarray*}
&&\int_{0}^{t}\frac{d}{ds}(x_1,x_2,x_3)(s)ds
 -\frac{1}{\omega}
 \int_{0}^{\omega}\int_{0}^{t}
 \frac{d}{dm}(x_1,x_2,x_3)(m)dm ds=(x_1,x_2,x_3)(t),
\end{eqnarray*}
which is valid only for all $(x_1,x_2,x_3)^T\in \dom L\cap\Ker P.$

\end{enumerate}
Thus, form (a), (b) and (c) we can note that  the projectors
$P$ and $Q$ defined on \eqref{eq:operators_P_Q} satisfy the
condition $(i)$ of the Proposition~\ref{prop:projectorsPQ} and from  (d)
the condition $(ii)$ of the Proposition~\ref{prop:projectorsPQ}
is also satisfied.

\subsubsection{$N$ defined on \eqref{eq:operator_N} is a continuous operator}
$ $

\begin{lemma}
\label{lem:N_is_contin}
Let us consider $X$ and $Y$ the Banach spaces given on \eqref{eq:X:space}.
Then $N:X\to Y$  defined on \eqref{eq:operator_L}
is a continuous operator.
\end{lemma}

\begin{proof}
Let us consider  that $\{(S^*_n,L^*_n,A^*_n)\}\subset X$ an arbitrary sequence
such that converges to $\{(S^*,L^*,A^*)\}$ in the norm induced topology of $X$.
Now, using that the inequality $|1-e^{-x}|\le x$ holds for all $x\in\R,$
we can prove that $|e^{-a}-e^{-b}|\le e^{-a}|a-b|$ and 
$|e^a-e^b|\le e^b|a-b|$. Then, by
by the definition of $N$ given on \eqref{eq:operator_N}
and by \eqref{eq:operator_N}-\eqref{eq:operator_N:3},
we deduce that there is $C>0$ depending on 
$b,\mu_1,\beta_1,\beta_2,\gamma_1,\gamma_2,\alpha_1$ and $\alpha_2$
such that
\begin{eqnarray*}
 \|N(S^*_n,L^*_n,A^*_n)-N(S^*,L^*,A^*)\|
 \le C \|(S^*_n,L^*_n,A^*_n)-(S^*,L^*,A^*)\|.
\end{eqnarray*}
Thus, the sequence $\{N(S^*_n,L^*_n,A^*_n)\}\subset X$ 
converges to $\{N(S^*,L^*,A^*)\}$
in the norm induced topology of $X$. Therefore, the operator
$N$ is continuous.
\end{proof}

\subsubsection{$N$ defined on \eqref{eq:operator_N} is $L$-compact on
any ball of $X$ centered at $(0,0,0)$}
$ $ 

\begin{lemma}
\label{lem:N_is_L_compact}
Let $h\in \mathbb{R}^+$ (a fix number) and consider that 
$\Omega\subset X$ is the open ball of radius $h$ and centered at $(0,0,0)$, i.e.
\begin{eqnarray}
\Omega =\{(x_1,x_2,x_3)\in X\;/\;\|(x_1,x_2,x_3)\|<h\}.
\label{eq:omega_set}
\end{eqnarray}
Moreover, consider that $L:\dom L\subset X\to Y$ and $N:X\to Y$ 
are the operators defined on \eqref{eq:operator_L} and \eqref{eq:operator_N},
respectively. If the hypothesis \eqref{eq:hipot_H} is satisfied,
then $N$ is $L$-compact on $\overline{\Omega}$.
\end{lemma}

\begin{proof}
We prove that $N,L$ and $\Omega$ satisfy the Definition~\ref{def:L_compact}.
Indeed, firstly we note that $\Omega $ is naturally an open bounded set by construction
and $L$ is a Fredholm operator of index zero by Lemma~\ref{lem:L_is_fredholm}.
Second, from \eqref{eq:operator_N} and \eqref{eq:operators_P_Q} we have that
$QN$ is defined by 
\begin{eqnarray}
QN((x_1,x_2,x_3)^T)=\frac{1}{\omega}\int_0^{\omega}
(\mathcal{N}_1(\tau),\mathcal{N}_2(\tau),\mathcal{N}_3(\tau))
d\tau
\label{eq:qn_operator}
\end{eqnarray}
where $\mathcal{N}_i,$ $i=1,2,3$, denote the functions
defined on \eqref{eq:operator_N:1}-\eqref{eq:operator_N:3}.
Moreover, for $(x_1,x_2,x_3)\in \overline{\Omega}$
and the hypothesis~\eqref{eq:hipot_H} we can prove
that the functions $\mathcal{N}_i$ are bounded.
Then, for $(x_1,x_2,x_3)\in \overline{\Omega}$ we have that
\begin{eqnarray*}
\|QN((x_1,x_2,x_3)^T)\|\le \frac{1}{\omega}\int_0^{\omega}
\|(\mathcal{N}_1,\mathcal{N}_2,\mathcal{N}_3)\|d\tau
=\|(\mathcal{N}_1,\mathcal{N}_2,\mathcal{N}_3)\|
\end{eqnarray*}
which implies that $QN(\overline{\Omega})$ is bounded.
Finally, from  \eqref{eq:operator_L}, \eqref{eq:operator_N} 
and \eqref{eq:operator_kp}, we have that $K_P(I-Q)N$ is defined
as follows
\begin{eqnarray*}
 &&(K_P(I-Q)N)((x_1,x_2,x_3)^T)(t)
 \\
 &&
 \qquad=
 \int_0^t(\mathcal{N}_1,\mathcal{N}_2,\mathcal{N}_3)(\tau)d\tau
 +\left(\frac{\omega^2}{2}-\frac{t}{\omega}\right)\int_0^{\omega}
	(\mathcal{N}_1,\mathcal{N}_2,\mathcal{N}_3)(\tau)d\tau
\\
&&
\qquad\quad
 -\frac{1}{\omega}\int_0^{\omega}\int_0^{\eta}
	(\mathcal{N}_1,\mathcal{N}_2,\mathcal{N}_3)(m)dmd\eta.
\end{eqnarray*}
Then, we deduce that
\begin{eqnarray*}
\|K_P(I-Q)N\|\le \left(\frac{\omega^2}{2}+3\omega\right)\|N\|,
\end{eqnarray*}
which implies that $(K_P(I-Q)N)(\overline{\Omega})$ is bounded,
since $N$ is bounded on $\overline{\Omega}$. Moreover,
for any $t,s\in [t_0,\infty[$,
we can deduce the following bound
\begin{eqnarray*}
|(K_P(I-Q)N)((x_1,x_2,x_3)^T)(t)-(K_P(I-Q)N)((x_1,x_2,x_3)^T)(s)|
\le 2\|N\||t-s|,
\end{eqnarray*}
i.e. the operator $K_P(I-Q)N$ is equicontinuous. Thus, by
application of Arzela Ascoli's theorem we have that 
$K_P(I-Q)N$ is a compact operator on $\overline{\Omega}$.
Therefore, the operator $N$  is  $L$-compact
on $\overline{\Omega}$.
\end{proof}

\subsubsection{The condition (C$_1$)-(C$_3$) of Theorem~\ref{teo:teocontinuacion}
are satisfied}
$ $ 

\begin{theorem}
\label{teo:contradiction}
Let us consider $X$ and $Y$ are the Banach spaces defined on \eqref{eq:X:space},
$Q$ defined by \eqref{eq:operators_P_Q}, and
$L:X\to Y$ and $N:X\to Y$ defined on \eqref{eq:operator_L} and \eqref{eq:operator_N},
respectively. Moreover, assume that the hypothesis \eqref{eq:hipot_H} is
satisfied. Then, there are the positive constants $\rho_1,\rho_2,\rho_3,d_1,d_2,d_3,
 \delta_1,\delta_2$ and $\delta_3$,
 such that the following assertions are valid
\begin{itemize}
 \item[(a)] If $\lambda\in (0,1)$ and $(x_1,x_2,x_3)\in \dom L $
 are such that $L(x_1,x_2,x_3)=\lambda N(x_1,x_2,x_3)$,
 the following inequalities  
\begin{eqnarray}
&&x_i(t)<\ln(\rho_i/w)+di,\quad i=1,2,3,
\label{eq:upbound}
\\
&&\ln(\delta_i)<x_i(t),\quad i=1,2,3, 
\label{eq:lowbound}  
\end{eqnarray}
holds for all $t\in [0,\omega].$
 
 \item[(b)] If $(x_1,x_2,x_3)\in \Ker L $
 are such that $QN(x_1,x_2,x_3)=0$,  the following inequalities  
\begin{eqnarray}
&&x_i(t)<\ln(\rho_i/w),\quad i=1,2,3,
\label{eq:upbound_ker}
\\
&&\ln(\delta_i)<x_i(t),\quad i=1,2,3, 
\label{eq:lowbound_ker} 
\end{eqnarray}
holds for all $t\in [0,\omega].$
 
\end{itemize}
 
\end{theorem}

\begin{proof}
{[(a)].} The proof is constructive and based on estimates for the system
associated to the operator equation $L=\lambda N$. Indeed, by the definition
of the operators $L$ and $N$, we notice that the system
$L(x_1(t),x_2(t),x_3(t))=\lambda N(x_1(t),x_2(t),x_3(t))$ is extensively written 
as follows
\begin{subequations}
\label{eq:virus_model_LlambdaN}
\begin{eqnarray}
\frac{dx_1}{dt}(t)
&=& \lambda\Big[b(t)\exp({-x_1(t)})
-\beta_{1}(t)\exp({x_2(t)})-\beta_{2}(t)\exp({x_3(t)})
\nonumber\\
&&
+\gamma_{1}(t)\exp\Big({x_2(t)-x_1(t)}\Big)
+\gamma_{2}(t)\exp\Big({x_3(t)-x_1(t)}\Big)- \mu_{1}(t)\Big],
\qquad\mbox{${}$}
\label{eq:virus_model_LlambdaN:s}\\
\frac{dx_2}{dt}(t)
&=&\lambda\Big[\beta_{1}(t)\exp({x_1(t)})
+ \beta_{2}(t)\exp\Big({x_1(t)+x_3(t)-x_2(t)}\Big)
\nonumber\\
&&
+\alpha_{2}(t)\exp\Big({x_3(t)-x_2(t)}\Big)
-[\mu_{2}(t)
+\alpha_{1}(t)+\gamma_{1}(t)]\Big],
\label{eq:virus_model_LlambdaN:l}\\
\frac{dx_3}{dt}(t)
&=&\lambda\Big[\alpha_{1}(t)\exp\Big({x_2(t)-x_3(t)}\Big)
-[\mu_{3}(t)+\alpha_{2}(t)+\gamma_{2}(t)]\Big].
\label{eq:virus_model_LlambdaN:a}
\end{eqnarray}
\end{subequations}
By the periodicity of $(x_1,x_2,x_3)$ we have that the integration of
\eqref{eq:virus_model_LlambdaN} on $[0,\omega]$ implies that
\begin{subequations}
\label{eq:virus_model_propo_1}
\begin{eqnarray}
&&\int_0^\omega b(t)\exp({-x_1(t)})dt= 
\int_0^\omega \Big[\mu_{1}(t)+\beta_{1}(t)\exp({x_2(t)})+\beta_{2}(t)\exp({x_3(t)})
\nonumber\\
&&
\hspace{1.0cm}
-\gamma_{1}(t)\exp\Big({x_2(t)-x_1(t)}\Big)
-\gamma_{2}(t)\exp\Big({x_3(t)-x_1(t)}\Big)\Big]dt,
\qquad\mbox{${}$}
\label{eq:virus_model_propo_1:s}\\
&&
\int_0^\omega\Big[\beta_{1}(t)\exp({x_1(t)})
+ \beta_{2}(t)\exp\Big({x_1(t)+x_3(t)-x_2(t)}\Big)
\nonumber\\
&&
\hspace{1cm}
+\alpha_{2}(t)\exp\Big({x_3(t)-x_2(t)}\Big)\Big]dt
=
\int_0^\omega\Big[
(\mu_{2}
+\alpha_{1}+\gamma_{1})(t)\Big]dt,
\hspace{.3cm}\mbox{{$ $}}
\label{eq:virus_model_propo_1:l}\\
&&
\int_0^\omega
\alpha_{1}(t)\exp\Big({x_2(t)-x_3(t)}\Big)dt
=
\int_0^\omega\Big[\mu_{3}(t)+\alpha_{2}(t)+\gamma_{2}(t)\Big]dt.
\label{eq:virus_model_propo_1:a}
\end{eqnarray}
\end{subequations}
Now, taking the modulus of the equations in the system  
\eqref{eq:virus_model_LlambdaN}, integrating the resulting equations
on $[0,\omega]$, using the fact that $\lambda\in (0,1)$
and the equations \eqref{eq:virus_model_LlambdaN}-\eqref{eq:virus_model_propo_1},
we obtain the following estimates
\begin{subequations}
\label{eq:virus_model_model_estimation}
\begin{eqnarray}
&&\int_0^\omega\left|\frac{dx_1}{dt}(t)\right| dt
< 2\int_0^\omega \Big[
\mu_{1}(t)+\beta_{1}(t)\exp({x_2(t)})+\beta_{2}(t)\exp({x_3(t)})\Big]dt,
\label{eq:virus_model_model_estimation:s}\\
&&\int_0^\omega\left|\frac{dx_2}{dt}(t)\right|dt
<2\int_0^\omega \Big[\mu_{2}(t)
+\alpha_{1}(t)+\gamma_{1}(t)\Big]dt,
\label{eq:virus_model_model_estimation:l}\\
&&\int_0^\omega\left|\frac{dx_3}{dt}(t)\right|dt
<2\int_0^\omega\Big[\mu_{3}(t)+\alpha_{2}(t)+\gamma_{2}(t)]\Big]dt.
\label{eq:virus_model_model_estimation:a}
\end{eqnarray}
\end{subequations}
We note that the right hand side of \eqref{eq:virus_model_model_estimation:l}
and \eqref{eq:virus_model_model_estimation:a} are finite as
consequence of the hypothesis \eqref{eq:hipot_H}. Meanwhile,
in order to get a finite bound of the right hand
side of \eqref{eq:virus_model_model_estimation:s} we need
some finite bounds for $\int_0^\omega\exp({x_2(t)})dt$ and $\int_0^\omega\exp({x_3(t)})dt$.

Multiplying the equations \eqref{eq:virus_model_LlambdaN:s}, 
\eqref{eq:virus_model_LlambdaN:l} and \eqref{eq:virus_model_LlambdaN:a}
by $\exp({x_1(t)}),\exp({x_2(t)})$ and $\exp({x_3(t)})$, respectively,
and then integrating on $[0,\omega]$,
we deduce that
\begin{subequations}
\label{eq:model_times_exp}
\begin{eqnarray}
&&\int_0^\omega b(t)dt= 
\int_0^\omega \Big[\mu_{1}(t)\exp(x_1(t))+\beta_{1}(t)\exp({x_1(t)+x_2(t)})
\nonumber\\
&&
\hspace{0.5cm}
+\beta_{2}(t)\exp({(x_1+x_3)(t)})
-\gamma_{1}(t)\exp\Big({x_2(t)}\Big)
-\gamma_{2}(t)\exp\Big({x_3(t)}\Big)\Big]dt,
\qquad\mbox{${}$}
\label{eq:model_times_exp:s}\\
&&
\int_0^\omega\Big[\beta_{1}(t)\exp({x_1(t)+x_2(t)})
+ \beta_{2}(t)\exp\Big({x_1(t)+x_3(t)}\Big)
+\alpha_{2}(t)\exp\Big({x_3(t)}\Big)\Big]dt
\nonumber\\
&&
\hspace{1cm}
=
\int_0^\omega\Big[
\mu_{2}(t)
+\alpha_{1}(t)+\gamma_{1}(t)\Big]\exp({x_2(t)})dt,
\label{eq:model_times_exp:l}\\
&&
\int_0^\omega
\alpha_{1}(t)\exp\Big({x_2(t)}\Big)dt
=
\int_0^\omega\Big[\mu_{3}(t)+\alpha_{2}(t)+\gamma_{2}(t)\Big]\exp({x_3(t)})dt.
\label{eq:model_times_exp:a}
\end{eqnarray}
\end{subequations}
Thus, we claim that
the positive behavior and the inequality satisfied by the coefficients 
given on hypothesis \eqref{eq:hipot_H}
and the system \eqref{eq:model_times_exp} imply the following bounds
\begin{subequations}
\label{eq:model_times_exp_bounds}
\begin{eqnarray}
&&\int_0^\omega \exp({x_1(t)})dt<
\frac{1}{\displaystyle\min_{t\in[0,\omega]}\mu_{1}(t)}
\left[
1
+\frac{\displaystyle\max_{t\in[0,\omega]}\gamma_{1}(t)}
{\displaystyle(1-\theta)\min_{t\in[0,\omega]}(\mu_{2}+\alpha_{1})(t)}
\right.
\nonumber\\
&&
\left.
\hspace{3.5cm}
+\frac{\theta\displaystyle\max_{t\in[0,\omega]}\gamma_{2}(t)}
{\displaystyle(1-\theta)\max_{t\in[0,\omega]}(\alpha_{2}+\gamma_2)(t)}
\right]
\int_0^\omega b(t)dt,
\label{eq:model_times_exp_bounds:s}\\
&&
\int_0^\omega \exp({x_2(t)})dt<
\frac{1}{(1-\theta)\displaystyle\min_{t\in[0,\omega]}
(\mu_{2}+\alpha_{1})(t)}
\int_0^\omega b(t)dt,
\label{eq:model_times_exp_bounds:l}\\
&&
\int_0^\omega \exp({x_3(t)})dt<
\frac{\theta}{(1-\theta)\displaystyle\max_{t\in[0,\omega]}
(\alpha_{2}+\gamma_{2})(t)}
\int_0^\omega b(t)dt,
\label{eq:model_times_exp_bounds:a}
\end{eqnarray}
\end{subequations}
where
\begin{eqnarray}
\theta=\max_{t\in[0,\omega]}
\left[\frac{\alpha_1}{\alpha_1+\mu_2}\right](t)
\quad
\frac{\displaystyle\max_{t\in[0,\omega]}(\alpha_2+\gamma_2)(t)}
{\displaystyle\min_{t\in[0,\omega]}(\mu_3+\alpha_2+\gamma_2)(t)}
\in ]0,1[
\cdot
\label{eq:notation_theta}
\end{eqnarray}
The fact that $\theta\in ]0,1[$ is a consequence of 
\eqref{eq:hipot_H}.
Now, we prove the inequalities 
\eqref{eq:model_times_exp_bounds:s}-\eqref{eq:model_times_exp_bounds:a}. 
Indeed, firstly we prove   \eqref{eq:model_times_exp_bounds:a}.
By \eqref{eq:model_times_exp:a},  we note that
 \begin{eqnarray}
\int_0^\omega \exp({x_3(t)})dt
&\le&\frac{1}{\displaystyle\min_{t\in[0,\omega]}(\mu_3+\alpha_2+\gamma_2)(t)}
\int_0^\omega (\mu_3+\alpha_2+\gamma_2)(t)\exp({x_3(t)})dt
\nonumber\\
&=&\frac{1}{\displaystyle\min_{t\in[0,\omega]}(\mu_3+\alpha_2+\gamma_2)(t)}
\int_0^\omega \alpha_1(t)\exp({x_2(t)})dt
\nonumber\\
&=&\int_0^\omega \frac{\alpha_1(t)}{(\alpha_1+\mu_2)(t)}
(\alpha_1+\mu_2)(t)\exp({x_2(t)})dt
\nonumber\\
&\le&
\displaystyle\frac{\displaystyle\max_{t\in[0,\omega]}
\left[\frac{\alpha_1}{\alpha_1+\mu_2}\right](t)}
{\displaystyle\min_{t\in[0,\omega]}(\mu_3+\alpha_2+\gamma_2)(t)}
\int_0^\omega 
(\alpha_1+\mu_2)(t)\exp({x_2(t)})dt,
\label{eq:cota1_exp3}
\end{eqnarray}
and by \eqref{eq:model_times_exp:s}-\eqref{eq:model_times_exp:l}
\begin{eqnarray}
&& 
\int_0^\omega(\alpha_{1}+\mu_{2})(t)\exp({x_2(t)})dt
\nonumber\\
&& 
\qquad=
\int_0^\omega\Big[\beta_{1}(t)\exp({x_1(t)+x_2(t)})
+ \beta_{2}(t)\exp\Big({x_1(t)+x_3(t)}\Big)
\nonumber\\
&& 
\qquad\qquad
+\alpha_{2}(t)\exp\Big({x_3(t)}-\gamma_{1}(t)\Big)\Big]dt
\nonumber\\
&& 
\qquad=
\int_0^\omega\Big[b(t)
+ \gamma_{2}(t)\exp\Big(x_3(t)\Big)
-\mu_{1}(t)\exp\Big({x_2(t)}\Big)
+\alpha_{2}(t)\exp\Big(x_3(t)\Big)\Big]dt
\nonumber\\
&& 
\qquad\le
\int_0^\omega b(t)dt
+\max_{t\in[0,\omega]}(\alpha_2+\gamma_2)(t)
\int_0^\omega \exp({x_3(t)})dt.
\label{eq:cota2_exp3}
\end{eqnarray}
Then, by \eqref{eq:notation_theta}, \eqref{eq:cota1_exp3} and \eqref{eq:cota2_exp3}
we deduce that
\begin{eqnarray}
\int_0^\omega \exp({x_3(t)})dt
&\le&
\displaystyle\frac{\theta}
{\displaystyle\min_{t\in[0,\omega]}(\alpha_2+\gamma_2)(t)}
\int_0^\omega b(t)dt
+\theta\int_0^\omega \exp({x_3(t)})dt,
\label{eq:cota3_exp3}
\end{eqnarray}
which implies  \eqref{eq:model_times_exp_bounds:a}.
Now,
from  \eqref{eq:cota2_exp3} and \eqref{eq:model_times_exp_bounds:a}
we get the following estimates
\begin{eqnarray}
\int_0^\omega\exp({x_2(t)})dt
&\le&
\frac{1}{\displaystyle\min_{t\in[0,\omega]}(\alpha_{1}+\mu_{2})(t)}
\int_0^\omega b(t)dt
\nonumber\\
&&
\hspace{0.5cm}
+\frac{\displaystyle\max_{t\in[0,\omega]}(\alpha_2+\gamma_2)(t)}
{\displaystyle\min_{t\in[0,\omega]}(\alpha_{1}+\mu_{2})(t)}
\int_0^\omega \exp({x_3(t)})dt
\nonumber
\\
&\le&
\left[
\frac{1}{\displaystyle\min_{t\in[0,\omega]}(\alpha_{1}+\mu_{2})(t)}
+\frac{\theta}{(1-\theta)\displaystyle
\min_{t\in[0,\omega]}(\alpha_{1}+\mu_{2})(t)
}
\right]
\int_0^\omega b(t)dt
\nonumber
\\
&=&
\frac{1}{(1-\theta)\displaystyle\min_{t\in[0,\omega]}(\alpha_{1}+\mu_{2})(t)
}
\int_0^\omega b(t)dt
\end{eqnarray}
and we prove   \eqref{eq:model_times_exp_bounds:l}.  
The inequality
\eqref{eq:model_times_exp_bounds:s}
is a consequence 
of \eqref{eq:model_times_exp:s} and the estimates
\eqref{eq:model_times_exp_bounds:l} and \eqref{eq:model_times_exp_bounds:a}, since
\begin{eqnarray}
&&\int_0^\omega\exp({x_1(t)})dt
\nonumber\\
&&
\qquad
\le
\frac{1}{\displaystyle\min_{t\in[0,\omega]}\mu_{1}(t)}
\int_0^\omega \mu_{1}(t)\exp({x_1(t)})
+\beta_{1}(t)\exp({(x_1+x_2)(t)})
\nonumber\\
&&
\hspace{3.5cm}
+\beta_{2}(t)\exp({(x_1+x_3)(t)})
dt
\nonumber
\\
&&
\qquad\le
\frac{1}{\displaystyle\min_{t\in[0,\omega]}\mu_{1}(t)}
\left[
\int_0^\omega b(t)dt
+\int_0^\omega(\gamma_{1}(t)\exp({x_2(t)})+\gamma_{2}(t)\exp({x_3(t)}))dt
\right]
\nonumber
\\
&&\qquad
\le
\frac{1}{\displaystyle\min_{t\in[0,\omega]}\mu_{1}(t)}
\nonumber\\
&&
\qquad
\times
\left[
1
+\frac{\displaystyle\max_{t\in[0,\omega]}\gamma_{1}(t)}
{\displaystyle(1-\theta)\min_{t\in[0,\omega]}(\mu_{2}+\alpha_{1})(t)}
+\frac{\theta\displaystyle\max_{t\in[0,\omega]}\gamma_{2}(t)}
{\displaystyle(1-\theta)\max_{t\in[0,\omega]}(\alpha_{2}+\gamma_2)(t)}
\right]
\int_0^\omega b(t)dt.
\nonumber
\end{eqnarray}
Thus the inequalities in \eqref{eq:model_times_exp_bounds} are proved.

Let us introduce some notation. Given
$f$, a positive real valued bounded function on $[0,\omega]$,   
we introduce the following notation 
\begin{eqnarray}
\overline{f}=\frac{1}{\omega}\int_0^\omega f(t)dt,
\quad
f^\bot =\min_{x\in [0,\omega]} f(t),
\quad
\mbox{and}
\quad
f^\top =\max_{x\in [0,\omega]} f(t).
\label{eq:notationavminmax}
\end{eqnarray}
Then by \eqref{eq:model_times_exp_bounds}, we follow
that there are the positive constants $\rho_1,\rho_2$ and $\rho_3$ given by
\begin{eqnarray}
\rho_1&=&\frac{\omega \overline{b}}{u_1^\bot}
\left[
1+\frac{\gamma_1^\top}{(1-\theta)(\mu_{2}+\alpha_{1})^\bot}
+\frac{\theta\gamma_2^\top}{(1-\theta)(\alpha_{2}+\gamma_2)^\top}
\right],
\label{eq:rho_1}\\
\rho_2&=& \frac{\omega \overline{b}}{(1-\theta)(\mu_{2}+\alpha_{1})^\bot},
\label{eq:rho_2}\\
\rho_3&=& \frac{\theta\omega \overline{b}}{(1-\theta)(\alpha_{2}+\gamma_2)^\top},
\label{eq:rho_3}
\end{eqnarray}
such that 
\begin{eqnarray}
\int_0^\omega\exp(x_i(t)) dt < \rho_i,\quad i=1,2,3.
\label{eq:expbounds_ii} 
\end{eqnarray}
Then, by application of the intermediate value for integrals 
in \eqref{eq:expbounds_ii} we get that
there are $\xi_i\in [0,\omega]$ such that $x_i(\xi_i)<\ln(\rho_i/\omega)$.
Thus, by the fundamental theorem of calculus 
and \eqref{eq:virus_model_model_estimation} and
\eqref{eq:expbounds_ii}
we deduce that there are the positive constants
$d_1,d_2$ and $d_3$ defined as follows
\begin{eqnarray}
d_1=2[\omega \overline{\mu_1}+\beta_1^\top\delta_2+\beta_2^\top\delta_3],
\;
d_2=2\omega\; \overline{\mu_2+\alpha_1+\gamma_1},
\;
d_3=2\omega \;\overline{\mu_3+\alpha_2+\gamma_2},
\label{eq:d_123}
\end{eqnarray}
such that
\begin{eqnarray}
x_i(t)=x(\xi_i)+\int_{\xi_i}^t\frac{dx}{dt}(t)dt
<
\ln(\rho_i/\omega)+\int_{\xi_i}^t\frac{dx}{dt}(t)dt
<\ln(\rho_i/\omega)+d_i,
\end{eqnarray}
which implies that the inequality \eqref{eq:upbound} is valid.

The proof of \eqref{eq:lowbound} is given as follows.
For $i=1,2,3$, let us denote by $\tau_i\in [0,\omega]$ the points where $x_i$
has a minimum, i.e.
$\displaystyle x_1(\tau_i)=\min_{t \in [0,\omega]}x_i(t).$
Then, from \eqref{eq:virus_model_LlambdaN} we have that
\begin{subequations}
\label{eq:vir_mod_min}
\begin{eqnarray}
0&=& b(\tau_1)\exp({-x_1(\tau_1)})
-\beta_{1}(\tau_1)\exp({x_2(\tau_1)})-\beta_{2}(\tau_1)\exp({x_3(\tau_1)})
\nonumber\\
&&
+\gamma_{1}(\tau_1)\exp\Big({x_2(\tau_1)-x_1(\tau_1)}\Big)
\nonumber\\
&&
+\gamma_{2}(\tau_1)\exp\Big({x_3(\tau_1)-x_1(\tau_1)}\Big)- \mu_{1}(\tau_1),
\qquad\mbox{${}$}
\label{eq:vir_mod_min:s}\\
0
&=&\beta_{1}(\tau_2)\exp({x_1(\tau_2)})
+ \beta_{2}(\tau_2)\exp\Big({x_1(\tau_2)+x_3(\tau_2)-x_2(\tau_2)}\Big)
\nonumber\\
&&
+\alpha_{2}(\tau_2)\exp\Big({x_3(\tau_2)-x_2(\tau_2)}\Big)
-[\mu_{2}(\tau_2)
+\alpha_{1}(\tau_2)+\gamma_{1}(\tau_2)],
\label{eq:vir_mod_min:l}\\
0
&=&\alpha_{1}(\tau_3)\exp\Big({x_2(\tau_3)-x_3(\tau_3)}\Big)
-[\mu_{3}(\tau_3)+\alpha_{2}(\tau_3)+\gamma_{2}(\tau_3)].
\label{eq:vir_mod_min:a}
\end{eqnarray}
\end{subequations}
Now, using the notation  \eqref{eq:notationavminmax}, from 
\eqref{eq:upbound} and \eqref{eq:vir_mod_min:s} we deduce
the following inequalities
\begin{eqnarray*}
b^\bot 
&<& 
b(\tau_1)+\gamma_1(\tau_1)\exp(x_2(\tau_1))+\gamma_2(\tau_1)\exp(x_3(\tau_1))
\\
&<& 
\Big[\mu_1(\tau_1)+\beta_1(\tau_1)\exp(x_2(\tau_1))+\beta_2(\tau_1)\exp(x_3(\tau_1))\Big]
\exp(x_1(\tau_1))
\\
&<& 
\Big[\mu_1^\top+\beta_1^\top\exp(x_2(\tau_1))+\beta_2^\top\exp(x_3(\tau_1))\Big]
\exp(x_1(\tau_1))
\\
&<& 
\Big[\mu_1^\top+\beta_1^\top\frac{\rho_2}{\omega}\exp(d_2)
+\beta_2^\top\frac{\rho_3}{\omega}\exp(d_3)\Big]
\exp(x_1(\tau_1)).
\end{eqnarray*}
Then 
\begin{eqnarray}
\exp(x_1(\tau_1))>\frac{\omega b^\bot}
 {\omega\mu_1^\top+\beta_1^\top\rho_2\exp(d_2)+\beta_2^\top \rho_3\exp(d_3)}.
\label{eq:low_bound_exp_1}
\end{eqnarray}
Similarly, from \eqref{eq:vir_mod_min:l} and \eqref{eq:vir_mod_min:a}
we get 
\begin{eqnarray}
\exp(x_2(\tau_2))>\frac{\beta_1^\bot}{(\mu_2+\alpha_1+\gamma_1)^\top},
 \quad\mbox{and}\quad
\exp(x_3(\tau_3))>\frac{\alpha_1^\bot}{(\mu_3+\alpha_3+\gamma_2)^\top}.
\label{eq:low_bound_exp_2}
\end{eqnarray}
Thus, we have that there are  $\delta_1,\delta_2$ and $\delta_3$
defined by
\begin{eqnarray}
 \delta_1&=&\frac{\omega b^\bot}
 {\omega\mu_1^\top+\beta_1^\top\rho_2\exp(d_2)+\beta_2^\top \rho_3\exp(d_3)},
 \label{eq:delta_1}
 \\
  \delta_2&=&\frac{\beta_1^\bot}{(\mu_2+\alpha_1+\gamma_1)^\top},
 \label{eq:delta_2}
 \\
  \delta_3&=&\frac{\beta_1^\bot}{(\mu_2+\alpha_1+\gamma_1)^\top},
 \label{eq:delta_3}
\end{eqnarray}
such that the estimate \eqref{eq:lowbound} is satisfied.

\vspace{1cm}
\noindent
[(b)]. If $(x_1,x_2,x_3)\in\Ker L$, then by \eqref{eq:kenelL} we have that 
$(x_1,x_2,x_3)=(s_0,l_0,a_0)\in\R^3$ is constant.
To fix ideas we consider that $(x_1,x_2,x_3)=(s_0,l_0,a_0)$.
Then, by \eqref{eq:qn_operator}
the condition $QN((x_1,x_2,x_3)^T)=QN((s_0,l_0,a_0)^T)=0$
implies that
\begin{subequations}
\label{eq:vir_mod_kerL}
\begin{eqnarray}
0&=& \overline{b}\exp({-s_0})
-\overline{\beta_{1}}\exp({l_0})-\overline{\beta_{2}}\exp({a_0})
+\overline{\gamma_{1}}\exp\Big({l_0-s_0}\Big)
+\overline{\gamma_{2}}\exp\Big({a_0-s_0}\Big)- \overline{\mu_{1}},
\qquad\qquad\mbox{${}$}
\label{eq:vir_mod_kerL:s}\\
0
&=&\overline{\beta_{1}}\exp({s_0})
+ \overline{\beta_{2}}\exp\Big({s_0+a_0-l_0}\Big)
+\overline{\alpha_{2}}\exp\Big({a_0-l_0}\Big)
-\overline{\mu_{2}+\alpha_{1}+\gamma_{1}},
\label{eq:vir_mod_kerL:l}\\
0
&=&\overline{\alpha_{1}}\exp\Big({l_0-a_0}\Big)
-\overline{\mu_{3}+\alpha_{2}+\gamma_{2}}.
\label{eq:vir_mod_kerL:a}
\end{eqnarray}
\end{subequations}
Then, from the system \eqref{eq:vir_mod_kerL}
and by similar to the given in the case (a), we 
can prove that the an inequality of the type 
\eqref{eq:expbounds_ii} is also valid 
in this case, i.e.
\begin{eqnarray*}
\exp(s_0)<\frac{\rho_1}{\omega},
\quad
\exp(l_0)<\frac{\rho_2}{\omega}
\quad
\mbox{and}
\quad
\exp(a_0)<\frac{\rho_3}{\omega}\cdot
\end{eqnarray*}
which implies \eqref{eq:upbound_ker}. Moreover, by using the fact that 
$\Ker L\subset\dom L$ we have that the estimates 
\eqref{eq:low_bound_exp_1}-\eqref{eq:low_bound_exp_1} are valid for 
$(s_0,l_0,a_0)$, i.e.
\begin{eqnarray*}
 \exp(s_0)>\delta_1,
\quad
\exp(l_0)>\delta_2
\quad
\mbox{and}
\quad
\exp(a_0)>\delta_3\cdot
\end{eqnarray*}
Thus, the inequality \eqref{eq:lowbound_ker} is
also satisfied.
\end{proof}

\begin{lemma}
\label{lem:c1_hasta_c3}
Let us consider $X$ and $Y$ are the Banach spaces defined on \eqref{eq:X:space};
$\Omega\subset X$ the open ball with radius
\begin{eqnarray}
 h=\sum_{i=1}^3\max\left\{
 \Big|\ln (\delta_i)\Big|,
 \left|\ln \left(\frac{\rho_i}{\omega_i}\right)\right|+d_i
 \right\}
\end{eqnarray}
with $\delta_i,\rho_i$ and $d_i$ defined on \eqref{eq:delta_1}-\eqref{eq:delta_3},
\eqref{eq:rho_1}-\eqref{eq:rho_3}, and \eqref{eq:d_123}, respectively. Let
$Q$ defined by \eqref{eq:operators_P_Q}, $L:X\to Y$ and $N:X\to Y$ the operators
defined on \eqref{eq:operator_L} and \eqref{eq:operator_N},
respectively. Moreover, assume that the hypothesis \eqref{eq:hipot_H} is
satisfied. Then, $L$ and $N$ satisfy the properties (C$_1$)-(C$_3$) of 
Theorem~\ref{teo:teocontinuacion}.

\end{lemma}

\begin{proof}
The proof of (C$_1$) and (C$_2$) is given by contradiction argument and
the proof of (C$_3$) is constructive and by using the invariance property
of the topological degree. More precisely, we have that  
\begin{enumerate}
 \item[(C$_1$)] Let us assume that 
there are $\lambda\in ]0,1[$
and $(x_1,x_2,x_3)\in \partial\Omega\cap\dom L$
such that $L(x_1,x_2,x_3)=\lambda L(x_1,x_2,x_3).$
Then, by application of  Theorem~\ref{teo:contradiction}-(a)
we deduce that $(x_1,x_2,x_3)\in \interior \Omega$ which is a
contradiction to the assumption that $(x_1,x_2,x_3)\in \partial\Omega.$

  \item[(C$_2$)] Let us assume that 
there is $(x_1,x_2,x_3)\in \partial\Omega\cap\Ker L$
such that $QN(x_1,x_2,x_3)=0.$
Then, by application of  Theorem~\ref{teo:contradiction}-(b)
we deduce that $(x_1,x_2,x_3)\in \interior \Omega$ which is a
contradiction to the assumption that $(x_1,x_2,x_3)\in \partial\Omega.$

\item[(C$_3$)] Let us consider  the mapping
 $\Phi:\dom L \times [0,1]\to X$ defined as follows
\begin{eqnarray*}
  \Phi(x_1,x_2,x_3,\epsilon)&=&\left[
   \begin{array}{l}
     \overline{b}\exp(-x_1)
     -\overline{\beta_1}\exp(x_2)
     -\overline{\beta_2}\exp(x_3)
     -\overline{\mu_1} \\
     \overline{\beta_1}\exp(x_1)
     - [\overline{\mu_2+\alpha_1+\gamma_1}]\\
     \overline{\alpha_1}\exp(x_2-x_3)-
     [\overline{\mu_3+\alpha_2+\gamma_2}] \\
   \end{array}
 \right]
 \\
 &&
 +\varepsilon \left[
            \begin{array}{l}
              \overline{\gamma_1}\exp(x_2-x_1)
              +\overline{\gamma_2}\exp(x_3-x_1) \\
              \overline{\beta_2}\exp(x_1+x_3-x_2)
              +\overline{\alpha_2}\exp(x_3-x_2) \\
              0 \\
            \end{array}
          \right].
\end{eqnarray*}
We prove that $\Phi(x_{1},x_{2},x_{3},\epsilon)\neq 0$ when
$(x_{1},x_{2},x_{3})^{T}\in
\partial \Omega \cap \Ker L$. Here, by  \eqref{eq:kenelL}
recall that 
$(x_{1},x_{2},x_{3})^T=(s_0,l_0,a_0)\in\R^3$ is
a constant vector.
Now, assuming that the conclusion is not true,
then the constant vector $(s_0,l_0,a_0)^{T}$ with
$\|(s_0,l_0,a_0)\|=h$ satisfies
$\Phi(s_0,l_0,a_0,\epsilon)=0$, that is,
\begin{eqnarray*}
\qquad\qquad
0&=& \overline{b}\exp(-s_0)-\overline{\beta_{1}}\exp(l_0)-
\overline{\beta_{2}}\exp(a_0)
\\
&&-\overline{\mu_{1}}
+\varepsilon[\overline{\gamma_{1}}\exp(l_0-s_0)
+\overline{\gamma_{2}}\exp(a_0-s_0)], \\
 0&=& \overline{\beta_{1}}\exp(s_0)-[\overline{\mu_{2}
   +\alpha_{1}+\gamma_{1}}]+\varepsilon[
\overline{\beta_{2}}\exp(s_0+a_0-l_0)
\\
&&
+\overline{\alpha_{2}}\exp(a_0-l_0)],\\
0 &=&  \overline{\alpha_{1}}\exp(l_0-a_0)-[\overline{\mu_{3}+\alpha_{2}+\gamma_{2}}]. 
\end{eqnarray*}
Then, by similar arguments 
to the proof Theorem~\ref{teo:contradiction}-(a)
we obtain that $\|(s_0,l_0,a_0)^{T}\|<
 h$, which contradicts to the assumption that $\|(s_0,l_0,a_0)^{T}\|=h$.

By the Homotopy Invariance Theorem of Topology Degree,
 taking 
  $J=I:\ima Q \to \Ker L$  such that 
  $ (x_{1},x_{2},x_{3})^{T}\mapsto (x_{1},x_{2},x_{3})^{T}$,
  using the fact that the system
 \begin{eqnarray*}
 \hspace{1cm}
     \overline{b}e^{-x_{1}(t)}-
\overline{\mu_{1}}-\overline{\beta_{1}}e^{x_{2}(t)}-
\overline{\beta_{2}}e^{x_{3}(t)}
+\overline{\gamma_{1}}e^{(x_{2}-x_{1})(t)}+\overline{\gamma_{2}}e^{(x_{3}-x_{1})(t)}
&=&0 \\
   \overline{\beta_{1}}e^{x_{1}(t)}+
\overline{\beta_{2}}e^{x_{1}(t)+x_{3}(t)-x_{2}(t)}
+\overline{\alpha_{2}}e^{x_{3}(t)-x_{2}(t)}
-[\overline{\mu_{2}+\alpha_{1}+\gamma_{1}}]&=&0\\
   \overline{\alpha_{1}}e^{x_{2}(t_{3})-x_{3}(t_{3})}
   -[\overline{\mu_{3}+\alpha_{2}+\gamma_{2}}]&=&0,
 \end{eqnarray*}
 has a unique solution 
 $(x_{1}^{\star},x_{2}^{\star},x_{3}^{\star})^{T}\in \partial\Omega\cap \Ker L,$
 and by definition~\ref{def:degree},  we have that
\begin{eqnarray*}
\quad
&&\mathrm{deg}(JQN(x_{1},x_{2},x_{3})^{T},\Omega \cap \Ker L, (0,0,0)^{T})
\\
&&
\quad
=\mathrm{deg}(\Phi(x_{1},x_{2},x_{3},1),\Omega \cap \Ker L, (0,0,0)^{T})
\\
&&
\quad
=\mathrm{deg}\Big((\overline{b}e^{-x_{1}}-\overline{\beta_{1}}e^{x_{2}}-
\overline{\beta_{2}}e^{x_{3}}-\overline{\mu_{1}},\overline{\beta_{1}}e^{x_{1}}
\\
&&
\qquad
-[\overline{\mu_{2}+\alpha_{1}+\gamma_{1}}],\overline{\alpha_{1}}e^{x_{2}-x_{3}}
-[\overline{\mu_{3}+\alpha_{2}+\gamma_{2}}])^{T},\Omega
\cap \Ker L, (0,0,0)^{T}\Big)
\\
&&
\quad
=\mathrm{sgn}\left|
      \begin{array}{ccc}
        -\overline{b}e^{-x_{1}^{\star}} & -\overline{\beta_{1}}e^{x_{2}^{\star}} & -
\overline{\beta_{2}}e^{x_{3}^{\star}} \\
        \overline{\beta_{1}}e^{x_{1}^{\star}} & 0 & 0 \\
        0 & \overline{\alpha_{1}} e^{x_{2}^{\star}-x_{3}^{\star}}& - \overline{\alpha_{1}} e^{x_{2}^{\star}-x_{3}^{\star}} \\
      \end{array}
    \right|
\\
&&
\quad
=\mathrm{sgn}\Big[-\Big(
\overline{\alpha_{1}}\;\overline{\beta_{1}}\;\overline{\beta_{2}}e^{x_{1}^{\star}
+x_{2}^{\star}}+\overline{\alpha_{1}}\;\overline{\beta_{1}}^{2}e^{x_{1}^{\star}
+2x_{2}^{\star}-x_{3}^{\star}}
\Big)\Big]
\\
&&
\quad
=-1,
\end{eqnarray*}
which implies that 
$\mathrm{deg}(JQN,\Omega \cap \Ker L, 0)\not=0$ and prove that  (C3) is valid.
\end{enumerate}
Thus, the properties (C$_1$)-(C$_3$) 
of the Theorem~\ref{teo:teocontinuacion} are valid for
the given operators.
\end{proof}

\subsubsection{Proof of Theorem~\ref{teo:exitencia_equiv}}
By Lemmata~\ref{lem:L_is_fredholm}, \ref{lem:N_is_contin}
\ref{lem:N_is_L_compact} and \ref{lem:c1_hasta_c3} we have that
the hypotheses of the Theorem~\ref{teo:teocontinuacion}
are valid. Then, the
operator equation \eqref{eq:operator_L} has at least
one solution in  $\dom L\cap\overline{\Omega}\subset X$
and naturally this fact implies that the system  \eqref{eq:virus_model_equiv}
has at least one  $\omega-$periodic solution.

\subsection{Final remarks of the proof of  Theorem~\ref{teo:main}}

If the hypothesis \ref{eq:hipot_H} is valid, by
Theorem~\ref{teo:exitencia_equiv} we have that 
the system  \eqref{eq:virus_model_equiv}
has at least one $\omega-$periodic solution. Thus, by 
applying the Theorem~\ref{prop:sist_eqiv}, we conclude
the proof of Theorem~\ref{teo:main}.

\section*{Acknowledgments}

A. Coronel and F. Huancas thanks for the  support of
research projects
 DIUBB GI 153209/C and DIUBB GI 152920/EF at 
Universidad del B{\'\i}o-B{\'\i}o, Chile.
M. Pinto thanks for the  support of FONDECYT 1120709.

 \end{document}